\documentclass[reqno]{amsart}

\usepackage{amsmath}
\usepackage{amsfonts}
\usepackage{amsthm}
\usepackage{amssymb}
\usepackage{graphicx}
\usepackage{graphics}
\usepackage{bm}
\usepackage{dsfont}
\usepackage{color}
\usepackage{enumerate}
\usepackage[font=footnotesize]{caption}
\usepackage[all]{xy}
\usepackage{xypic}
\allowdisplaybreaks

\numberwithin{equation}{section}

\newtheorem{thm}{Theorem}[section]
\newtheorem{cor}[thm]{Corollary}
\newtheorem{lem}[thm]{Lemma}
\newtheorem{prop}[thm]{Proposition}
\theoremstyle{definition}

\newtheorem{dfn}[thm]{Definition}
\newtheorem{rmk}[thm]{Remark}

\newcommand{\N}{\mathds{N}}

\newcommand{\Z}{\mathds{Z}}

\newcommand{\R}{\mathds{R}}

\begin{document}

\title[Morse homology and elliptic pde's]{Morse homology for a class of elliptic partial differential equations}

\author[L. Asselle]{Luca Asselle}
\address{Ruhr-Universit\"at Bochum, Universit\"atsstra\ss e 150, 44801, Bochum, Germany}
\email{luca.asselle@rub.de}

\author[S. Cingolani]{Silvia Cingolani}
\address{Universit\'a degli Studi di Bari Aldo Moro, Via Orabona 4, 70125 Bari, Italy}
\email{silvia.cingolani@uniba.it}

\author[M. Starostka]{Maciej Starostka}
\address{Gda\'nks University of Technology, Gabriela Narutowicza 11/12, 80233 Gda\'nsk, Poland}
\email{maciej.starostka@pg.edu.pl}

\date{\today}

\begin{abstract}
In this paper we show that a notion of non-degeneracy which allows to develop Morse theory is generically satisfied for a large class of $C^2$-functionals defined on Banach spaces.
The main element of novelty with respect to the previous work \cite{Asselle:2024} of the first and third author is that 
we do not assume the splitting induced by the second differential at a critical point to persist in a neighborhood, 
provided one can give precise estimates on how much persistence fails. This allows us to enlarge significantly the class of elliptic 
pde's for which non-degeneracy holds and Morse homology can be defined. A concrete example is given by
equations involving the $p$-Laplacian, $p\leq n$. As a byproduct, we provide a criterion of independent interest to check whether critical 
points are non-degenerate in the sense above, 
and give an abstract construction of Morse homology in a Banach setting for functionals satisfying the Cerami condition. 
\end{abstract}

\maketitle

\textbf{Data availability statement:} Data sharing not applicable to this article as no datasets were generated or analysed during the current study.

\vspace{1mm}

\textbf{Conflict of interests statement:} On behalf of all authors, the corresponding author states that there is no conflict of interest. 





\section{Introduction}
\label{sec:introduction}

Let $\Omega\subset \R^n$, $n\geq 2$, be an open bounded set with boundary of class $C^1$. For $p>2$ fixed, we set $X:=W^{1,p}_0(\Omega)$ and consider a functional $f:X\to \R$ of the form 
\begin{equation}
f(u) := \int_\Omega \Psi(\nabla u(x))\, \mathrm d x - \int_\Omega G(x,u(x))\, \mathrm d x, \quad \forall u \in X,
\label{eq:functionalf}
\end{equation}
where $G(x,s) := \int_0^s g(x,t) \, \mathrm d t$ for some function $g:\Omega \times \R\to \R$ such that $g(\cdot,s)$ is measurable for every $s\in \R$ and $g(x,\cdot)$ is of class $C^1$ for a.e. $x\in \Omega$, and $\Psi:\R^n\to \R$ is of class $C^2$ with $\Psi(0)=0$ and $\nabla \Psi(0)=0$. The typical example we are interested in is $G(x,s)=G(s) = \int_0^s g(t) \, \mathrm d t$ and 
$$\Psi_{\kappa} : \R^n\to \R, \quad \Psi_{\kappa}(x) := \frac 1p \Big [\big (\kappa^2 + |x|^2\big)^{p/2} - \kappa^2\Big ]$$
for some $\kappa>0$. Notice that, for $\kappa=1$, critical points of the associated $f$ given by \eqref{eq:functionalf} correspond to non-trivial solutions of the following quasilinear elliptic problem involving the $p$-area functional: 

$$\left \{ \begin{array}{r} - \text{div}\, \Big [\big (1 + |\nabla u|^2\big )^{\frac{p-2}2} \nabla u \Big ] = g(u) \quad \text{in}\ \Omega, \\ u=0 \qquad \qquad \qquad \qquad  \ \ \ \text{on}\ \partial \Omega.\end{array}\right . $$

More generally, we make the following assumptions on the functions $g$ and $\Psi$. Here, as usual $p^*:= \frac{np}{n-p}$ denotes the Sobolev conjugate exponent, and we set $p^*=+\infty$ if $p\geq n$.

\begin{itemize}
\item[($\Psi$)] There exist $\kappa,\mu_1,\mu_2>0$ such that the functions 
$$\Psi-\mu_1\Psi_\kappa,\ \mu_2\Psi_\kappa - \Psi$$
are both convex.
 \item[($g$)] There exist $c>0$ and $0\leq q<p^*-2$ such that 
 $$|\partial_s g(x,s)| \leq c( 1+ |s|^{q}), \quad \text{for a.e.}\ x\in \Omega, \ \forall s\in \R.$$
\end{itemize}

Notice that Assumption ($\Psi$) implies in particular that $\Psi$ is strictly convex, whereas Assumption ($g$) implies that $f$ is of class $C^2$ in $X$.  Following \cite{Cingolani:2018} we see that critical points of $f$ correspond to weak solutions of  
\begin{equation}
\label{eq:bvp}
\left \{ \begin{array}{r} - \text{div}\, \big [\nabla \Psi (\nabla u)\big ] = g(x, u) \quad \quad  \text{in}\ \Omega, \\ u=0 \qquad \qquad  \qquad \ \ \ \text{on}\ \partial \Omega.\end{array}\right .\end{equation}

Equations as in \eqref{eq:bvp} appear ubiquitously in the mathematical description of many phenomena of physical origin. For instance, the choice of $\Psi(x):= \frac 1p |x|^p + \frac 12 |x|^2$ yields to the quasilinear problem 
$$\left \{ \begin{array}{r} - \Delta_p u - \Delta u =  g(x,u) \quad \quad  \text{in}\ \Omega, \\ u=0 \qquad \qquad  \qquad \ \ \ \text{on}\ \partial \Omega.\end{array}\right .$$ 
which models the propagation of solitary waves, see \cite{Benci,Benci:1998}. Here, $\Delta_p u:=\text{div}\, (|\nabla u|^{p-2} \nabla u)$ is the $p$-Laplace operator. 
Such equations also model several phenomena in non-Newtonian mechanics, nonlinear elasticity and glaciology, combustion theory, population biology; see \cite{Antonini:2025,Diaz:1994,Glowing:2003}.

The fact that solutions of \eqref{eq:bvp} admit a variational characterization as critical points of the functional $f$ in \eqref{eq:functionalf} opens up the possibility to use Morse theoretical methods 
to obtain existence and multiplicity results of solutions. However, extending Morse theory to Banach spaces presents some conceptual difficulties (c.f. \cite{Chang:1993ng,Uhlenbeck:1972,Tromba:1977}).
In particular, the classical definition of {\sl non-degeneracy} for critical points does not seem reasonable in a Banach setting,
and several obstructions arise due to the lack of Fredholmness of the linearized operators at the critical points. 
Therefore the classical Morse Lemma, its generalized versions due to
Gromoll-Meyer \cite{Gromoll:1969jy}, and perturbation results \'a la 
Marino-Prodi \cite{Marino:1975ii} cannot be applied directly.
In \cite{Cingolani:2005,Cingolani:2018,Cingolani:2003}, some ideas are introduced to
study the critical groups of the Euler functional \eqref{eq:functionalf}.
Even though the second differential at a critical point $u$ is not bijective and, more precisely, 
not a Fredholm operator,  the presence of the semilinear term gives rise to
a linearized operator that is non-degenerate in the sense of elliptic partial
differential equations, and thus
one can perform a suitable finite dimensional reduction. 
In this way, it is possible to give a good extension of local Morse theory
for functionals as in \eqref{eq:functionalf}. Also, for such a specific class of functionals, the injectivity of
$f''(u)$ seems to yield a meaningful non-degeneracy notion. Indeed, one can show that, if  $f''(u)$ is injective, then $u$ is isolated and
the critical groups at $u$ can be computed by means of its Morse index.
Moreover, in \cite{Cingolani:2007} the authors proved that injectivity is generically satisfied, showing that each critical point $u$ of $f$ such that $f''(u)$  is not injective can be ``resolved'' in a finite number of 
critical points (of a locally approximating functional) at which the second differential is injective. 

%

Despite the aforementioned results, the problem of non-degeneracy still remains in a Banach setting. In fact, requiring the second differential at critical points to be injective 
is in general not enough to do Morse theory, because it does not even imply that critical points are isolated. Consider for instance the $C^2$-functional 
\begin{equation}
\varphi:\ell^2(\N)\to \R, \quad \varphi(v) := \sum_{n=1}^{+\infty} \frac{\cos (n v_n)}{n^4}.
\label{eq:counterexample}
\end{equation}
We readily see that $v=0$ is a critical point of $\varphi$ and that the second differential at $v=0$  is given by
$$\mathrm d^2 \varphi(0)[w,z] = \sum_{n=1}^{+\infty} \frac{w_nz_n}{n^2}.$$
In particular $\mathrm d^2 \varphi(0)$ is injective and positive definite, but $0$ is not isolated in the set of critical points of $\varphi$. 

In the literature, several notion of non-degeneracy for functionals defined on a Banach space have been introduced (see for instance the works of Uhlenbeck \cite{Uhlenbeck:1972}, Tromba \cite{Tromba:1977} and Chang \cite{Chang:1993ng}), however it is as of now not clear whether such notions of non-degeneracy are generically satisfied (and are not even known to hold in any concrete example). 
The notion of non-degeneracy we use here is inspired by \cite{Abbondandolo:2006lk}.

\begin{dfn}
Let $Y$ be a Banach space, and let $F:Y\to \R$ be a functional of class $C^1$. A critical point $y_0$ of $F$ is called \textit{non-degenerate} if there exist a neighborhood $\mathcal U\subset Y$ of $y_0$ and a bounded linear hyperbolic operator $L:Y\to Y$ such that $F$ is a Lyapounov function on $\mathcal U$ for the linear flow induced by $L$. We call $F$ a \textit{Morse function} if all its critical points are non-degenerate. 
\label{def:nondegintro}
\end{dfn}

A non-degenerate critical point in the sense above is clearly isolated. Also, such a notion of non-degeneracy is much weaker than the usual one even in a Hilbert setting 
(in fact, it is easy to construct examples of non-degenerate critical points in the sense above for which the Hessian is neither injective nor surjective), and, as we shall recall in Section 2, is the one to prefer for Morse theory, because it implies that stable and unstable manifolds of critical points are well-defined submanifold of $X$ of the expected dimension\footnote{At this point, there is a subtlety when dealing with functionals whose critical points having infinite Morse index and co-index. Such an issue will be ignored in this paper because it never occurs for functionals as in \eqref{eq:functionalf}.}. We shall stress the fact that the approach to Morse theory based on stable and unstable manifolds has a huge advantage over the classical one based on deformation of sublevel sets, critical groups, etc. because it can (a priori) be used also for functionals whose critical points have infinite Morse index and co-index, case in which critical groups always vanish.

In the previous work \cite{Asselle:2024} by the first and third author it 
is shown that critical points of functionals as in \eqref{eq:functionalf} with $\Psi(x) = \frac 1p |x|^p + \frac 12 |x|^2$, $p>n$, and $q<p-2$ are generically non-degenerate in the sense of Definition \ref{def:nondegintro}. Here, we build on such a result to prove the following

\begin{thm}
Let $f:X\to \R$ be a functional as in \eqref{eq:functionalf} satisfying $(\Psi)$ and $(g)$. Let $\bar u \in X$ be a critical point of $f$ such that $\mathrm d^2f(\bar u):X\to X^*$ is injective. Then $\bar u$ is non-degenerate in the sense of Definition \eqref{def:nondegintro}. 
\label{thm:main1}
\end{thm}

Adapting the argument in \cite{Cingolani:2007} one sees that, generically, the second differential is injective at every critical point $\bar u$ of $f$. As a corollary of Theorem \ref{thm:main1}, critical points of $f$ as in \eqref{eq:functionalf} are generically non-degenerate. 

The extension to $p\leq n$ is the most challenging step in the proof of Theorem \ref{thm:main1} and the major element of novelty of the present paper. 
 The key step, which is the main content of Section \ref{sec:hyperbolicoperator}, consists in showing that the desired hyperbolic operator still exists even if 
 the splitting $X=X^+\oplus X^-$ induced by $\mathrm d^2f(\bar u)$ does not persist in a neighborhood of $\bar u$, provided one can give precise estimates on how much persistence fails, see Lemma \ref{lem:p*-2}. Abstracting from the concrete situation of the present paper, we also state in Theorem \ref{thm:criterionnondeg} a criterion to check non-degeneracy for critical points of $C^2$-functionals defined on a Banach space. 

Having (generically) non-degeneracy, one can then construct Morse homology for functionals $f$ as in \eqref{eq:functionalf} following a standard scheme. In the literature however, Morse homology is always constructed assuming the Palais-Smale condition, and it is well-known (see Section 4 for an overview) that, in several cases, functionals as in \eqref{eq:functionalf} do not satisfy the Palais-Smale condition but rather the weaker Cerami condition.  
For this reason, in Section \ref{sec:MorseCerami} we give a sketch of the construction of Morse homology in an abstract Banach framework assuming the Cerami condition only. 

\begin{thm}
Let $f:X\to \R$ be a functional as in \eqref{eq:functionalf} satisfying $(\Psi)$ and $(g)$ and such that $\mathrm d^2f(\bar u):X\to X^*$ is injective for every $\bar u \in \text{Crit}\, (f)$. Assume that $f$ satisfies the Cerami condition, and choose $P\subset X$ open subset which is positively invariant under the flow of a negative gradient-like vector field $V$  for $f$ (for instance, $P=f^{-1}(-\infty,a)$ for some $a\in \R$) and such that $f$ is bounded from below on $X\setminus P$. Then, Morse homology with $\Z_2$-coefficients for $f$ is well-defined and isomorphic to $H_*(Y,P;\Z_2)$. 
\label{thm:main2}
\end{thm}

\vspace{2mm}

 \textbf{Acknowledgments:} L.A. is partially supported by the DFG-grant 540462524 {\sl \lq\lq Morse theoretical methods in Analysis, Dynamics and Geometry"}. S.C. is supported by PRIN PNRR  P2022YFAJH {\sl \lq\lq Linear and Nonlinear PDEs: New directions and applications''} (CUP H53D23008950001), and partially supported by INdAM-GNAMPA.   M.S. is partially supported by the NCN grant 2023/05/Y/ST1/00186 {\sl \lq\lq Morse theoretical methods in Analysis, Dynamics and Geometry''}. 
 



\section{Morse homology and the Cerami condition}
\label{sec:MorseCerami}

In this section we quickly recall the construction of Morse homology in an abstract Banach setting when the Morse index is finite, referring to \cite{Asselle:2024} and \cite{Abbondandolo:2006lk} for further details. The element of novelty here is that, in contrast with the cited references, we only assume the Cerami condition.

\begin{dfn}
Let $Y$ be a Banach space, and let $F:Y\to \R$ be a functional of class $C^1$. A critical point $y_0$ of $F$ is called \textit{non-degenerate} if there exist a neighborhood $\mathcal U\subset Y$ of $y_0$ and a bounded linear hyperbolic operator $L:Y\to Y$ such that $F$ is a Lyapounov function on $\mathcal U$ for the linear flow induced by $L$. We call $F$ a \textit{Morse function} if all its critical points are non-degenerate. 
\label{def:nondeg}
\end{dfn}


\begin{rmk}
In Morse theory one usually requires that the second differential $\mathrm d^2F(y_0):Y\to Y^*$ be an isomorphism. As it easily seen, such a requirement is meaningless in a Banach setting because the dual space $Y^*$ is in general not isomorphic to $Y$. Also, the weaker assumption that $\mathrm d^2F(y_0)$ be injective is in general not enough to do Morse theory as it does not even imply that the critical points are isolated. 
\end{rmk}

We recall that a bounded linear operator $L$ is called \textit{hyperbolic} if its spectrum is disjoint from the imaginary axis, i.e. $\sigma(L)\cap i\R = \emptyset$. The operator $L$ induces a splitting of $Y$ 
into $L$-invariant subspaces
$$Y= Y^+\oplus Y^-,$$
where $\sigma (L|_{Y^{\pm}}) = \sigma (L) \cap \{ \pm \text{Re} (z) >0\}$, and we call 
$$\mu_-(y_0) := \dim Y^- \in \N_0\cup \{+\infty\}, \quad \mu_+(y_0) := \dim Y^+ \in \N_0\cup \{+\infty\}$$
the \textit{Morse index} resp. \textit{co-index} of $y_0$. 

A Morse function $F:Y\to \R$ always admits a (negative) gradient-like vector field $V$ whose regularity depends on the regularity of the partition of unity\footnote{For our purposes, $Y$ will always admit smooth partitions of unity (this is indeed the case if $Y$ is reflexive, see \cite{Fry:2002}).}. The construction goes as follows\footnote{In case of infinite Morse indices extra care is needed in order not to lose crucial compactness properties.}: in a neighborhood $\mathcal U$ of a critical point $y_0$ we define $V$ to be the linear vector field $V(x) := Lx$, where $L$ is the hyperbolic operator appearing in the definition of non-degeneracy. If instead $y\in Y$ is not a critical point, then we choose a sufficiently small neighborhood $\mathcal U_y$ and a vector $v\in Y$ such that 
$$\|v\| \leq 2 \|\mathrm dF(y)\|, \quad \mathrm dF(z)[v]\leq -  \|\mathrm dF(z)\|^2, \forall z\in \mathcal U_y$$ 
(this is possible since $\mathrm dF$ is continuous) and then define $V(z):= v$ for all $v\in \mathcal U_y$. By means of a partition of unity we then 
glue all these local vector fields to obtain a global gradient-like vector field $V$\footnote{The vector field $V$ will be in general only gradient-like, even though it is by construction a pseudo-gradient for $F$ outside a neighborhood of the set of critical points. Also, to get a complete flow it might be necessary to replace $V$ with $V/\sqrt{1+\|V\|^2}$.}. 

In particular, unstable and stable manifold $W^{s}(y_0,V)$ and $W^s(y_0,V)$ of a critical point $y_0\in Y$ are well-defined open embedded submanifolds\footnote{This actually does not follow directly from non-degeneracy, as non-degeneracy in general does not imply local closedness of stable and unstable manifolds. In this paper we can ignore such an issue, since we are dealing with critical points having finite Morse index only. For more details we refer to Theorem 1.20 in \cite{Abbondandolo:2006lk}.} of $Y$ homeomorphic to open disks of dimension equal the Morse index and co-index of $y_0$ respectively. 

\begin{dfn}
Let $Y$ be a Banach space, and let $F:Y\to \R$ be a functional of class $C^1$. We say that $F$ satisfies the \textit{Cerami condition} if the following holds: any sequence $(y_n)\subset Y$ such that 
$$f(y_n) \to c, \quad \|\mathrm dF(y_n)\| (1+ \|y_n\|) \to 0, \quad \text{for}\ c\in \R, \ n \to +\infty,$$
admits a converging subsequence. 
\end{dfn}

The Cerami condition implies the Palais-Smale condition on bounded sets but is in general weaker than the global Palais-Smale condition, since it requires precompactness of a smaller set of sequences. As it is well-known, there is no global Morse theory (in any of its forms) for functionals satisfying Palais-Smale only on bounded sets. In the following proposition we show instead that the Cerami condition implies that the intersection between stable and unstable manifolds of pairs of critical points is contained in a bounded region, which is a crucial step towards the construction of Morse homology. 

\begin{rmk}
The Cerami condition implies that, for every $-\infty<a<b<+\infty$, the set 
$$\text{Crit}\, (F) \cap F^{-1}([a,b])$$
is compact. Indeed, any sequence $(y_n)\subset \text{Crit}\, (F) \cap F^{-1}([a,b])$ is a Cerami sequence (possibly after passing to a subsequence) and as such admits a converging subsequence. 
\end{rmk}

\begin{prop}
Let $F:Y\to \R$ be a Morse function satisfying the Cerami condition, and let $V$ be a (negative) gradient-like vector field for $F$ as above. Then, for every $-\infty<a<b<+\infty$ there exists $R=R(a,b)>0$ such that any flow-line of $V$ contained in $F^{-1}([a,b])$ stays completely in $B_R(0)$, the (open) ball of radius $R$ in $Y$. In other words, any $u:\R\to Y$ such that 
$$\dot u(t) = V(u(t)), \quad \forall t\in \R,$$
and $F(u(t))\in [a,b]$ for all $t\in \R$, satisfies $\|u(t)\|<R$ for all $t\in \R$. 
\label{prop:1}
\end{prop}

\begin{cor}
Let $y_0,y_1$ be critical points of the Morse function $F:Y\to \R$, and let $V$ be a (negative) gradient-like vector field for $F$ as above. Then, there exists $R>0$ such that 
$$W^u(y_0,V) \cap W^s(y_1,V) \subset B_R(0).$$
\label{cor:boundedness}
\end{cor}

\begin{proof}[Proof of Proposition \ref{prop:1}]
The Cerami condition implies that there exist $\epsilon>0$ and $r_0>0$ such that 
\begin{equation}
(1+\|y\|)\|\mathrm dF(y)\| >\epsilon, \quad \forall y \in F^{-1}([a,b]) \cap B_{r_0}(0)^c.
\label{eq:Cerami1}
\end{equation}
In particular, the region $F^{-1}([a,b]) \cap B_{r_0}(0)^c$ does not contain critical points of $F$, and this implies that $V$ is a pseudo gradient for $f$ on such a region\footnote{Here for simplicity we ignore the fact that flow of $V$ could not be globally defined. The argument however can be easily adapted replacing $V$ with $V/\sqrt{1+\|V\|^2}$.}. 

Let now $u:\R\to Y$ be a flow-line for $V$ entirely contained in $F^{-1}([a,b])$. Assume that $t_0<t_1\in \R$ are such that $\|u(t_0)\|=r_0$ and $u|_{[t_0,t_1]}\subset B_{r_0}(0)^c$.  Then we compute for $t\in [t_0,t_1]$: 
\begin{align*}
\|u(t)\| - r_0 &\leq \|u(t)-u(t_0)\| \leq \int_{t_0}^t \|\dot u(s)\| \, \mathrm d s = \int_{t_0}^t \|V(u(s))\| \, \mathrm ds \leq \frac 12 \int_{t_0}^t \| \mathrm d F(u(s))\| \, \mathrm ds.
\end{align*}
Using \eqref{eq:Cerami1} and the definition of $V$ we infer that 
\begin{align*}
\|u(t)\| &\leq r_0 + \frac 12 \int_{t_0}^t \| \mathrm d F(u(s))\| \, \mathrm ds\\
	&\leq r_0 + \frac1{2\epsilon} \int_{t_0}^t \|\mathrm dF(u(s))\|^2 (1+ \|u(s)\|)\, \mathrm ds\\
	&\leq r_0 -  \frac1{2\epsilon} \int_{t_0}^t (1+ \|u(s)\|) \mathrm dF(u(s))[V(s)] \, \mathrm ds\\
	&= r_0 - \frac 1{2\epsilon} \int_{t_0}^t\mathrm dF(u(s))[V(s)] \, \mathrm ds - \frac 1{2\epsilon} \int_{t_0}^t  \|u(s)\| \mathrm dF(u(s))[V(s)] \, \mathrm ds\\
	&= r_0 + \frac 1{2\epsilon} \big (F(u(t_0)) - F(u(t))\big ) + \frac 1{2\epsilon} \int_{t_0}^t  \|u(s)\| \Big (- \frac{\mathrm d}{\mathrm d s} \big (F\circ u\big ) (s)\Big ) \, \mathrm ds\\
	&= \alpha(t) + \int_{t_0}^t \|u(s)\| \beta (s) \, \mathrm ds,
	\end{align*}
where 
$$\alpha(t) := r_0 + \frac 1{2\epsilon} \big (F(u(t_0)) - F(u(t))\big ) ,\quad \beta(t) := - \frac 1{2\epsilon} \frac{\mathrm d}{\mathrm d t} \big (F\circ u\big ) (t).$$
Since $\alpha$ is non-decreasing and $\beta$ is non-negative, Gr\"onwall's lemma implies that 
\begin{align*}
\|u(t)\| &\leq \alpha(t) \text{exp}\Big ( \int_{t_0}^t \beta(s) \, \mathrm ds\Big ) \\
	&= \big ( r_0 + \frac 1{2\epsilon} \big (F(u(t_0)) - F(u(t))\big )\big ) \text{exp} \Big ( \frac 1{2\epsilon} \big (F(u(t_0)) - F(u(t))\big )\Big ) \\
	&\leq \big ( r_0 + \frac 1{2\epsilon} (b - a)\big ) \text{exp} \Big ( \frac 1{2\epsilon} (b-a)\Big ) =:R,
\end{align*}
where we have used the assumption that $u(\cdot)\subset F^{-1}([a,b])$. This completes the proof. 
\end{proof}

Under the additional assumption that the Morse index of critical points of the Morse function $F$ is finite, Corollary \ref{cor:boundedness} implies that 
$$W^u(y_0,V) \cap W^s(y_1,V)$$
is precompact for any pair of critical points $(y_0,y_1)$; see \cite[Proposition 2.2]{Abbondandolo:2006lk} for further details. Once one has precompactness, the construction of Morse homology follows by standard arguments.

 After a generic perturbation of $V$ we can assume that the pair $(F,V)$ satisfies the Morse-Smale condition up to order 2, meaning that the intersection between stable and unstable manifolds of pairs of critical points is transverse\footnote{Notice that the regularity of $V$ (and not of $F$) is important to achieve transversality.} whenever the difference of Morse indices is less than  or equal to 2, see \cite{Asselle:2024}. Therefore, for any $y_0,y_1\in \text{Crif}\, (F)$ such that $\mu_-(y_0)-\mu_-(y_1)\leq 2$, we have that 
$$W^u(y_0,V) \cap W^s(y_1,V)$$
is a precompact manifold of dimension $\mu_-(y_0)-\mu_-(y_1)$. In the particular case $\mu_-(y_0)-\mu_-(y_1)=1$ it is easy to see that $W^u(y_0,V) \cap W^s(y_1,V)$ consists of finitely many flow-lines. 

Assume now that $P\subset Y$ is an open subset which is positively invariant under the flow of $V$ and such that $F$ is bounded from below on $Q:= Y\setminus P$ (for instance, we could take $P=F^{-1}(-\infty,a)$ for some $a\in \R$)\footnote{If $F$ is bounded from below on $Y$, then we can simply take $P=\emptyset$.}. We can now define a chain complex by setting 
$$C_k(F,Q): = \!\!\!\!\!\!  \bigoplus_{\footnotesize y\in \text{Crit}_Q (F), \ \mu_-(y)=k} \!\!\!\!\!\! \Z_2 \cdot \langle y\rangle ,\quad \forall k \in \N_0,$$
where $\text{Crit}_Q(F) := \text{Crit}\, (F) \cap Q$, and 
$$\partial_k: C_k(F,Q)\to C_{k-1}(F,Q), \quad \partial_k y := \!\!\!\!\!\!  \sum_{z\in \text{Crit}_Q (F), \ \mu_-(z)=k-1} \!\!\!\!\!\!  \big (\# \{ \text{c.c. of}\ W^{u}(y,V)\cap W^s(z,V) \} \  \text{mod} \, 2\big ) \cdot z.$$
The assumption that $F$ be bounded from below on $Q$ ensures that the sum in the definition of $\partial_k$ is a finite sum. Moreover, the Morse-Smale condition up to order 2 implies that $\partial_k \circ \partial_{k-1}=0$ for all $k\in \N$, that is, that $\partial_*$ is a boundary operator. Clearly, the chain complex depends on the choice of the gradient-like vector field $V$. In contrast, the resulting homology is independent of the choice of $V$ 
and is indeed isomorphic to the (relative) singular homology $H_*(Y,P;\Z_2)$, see \cite{Abbondandolo:2006lk}. 

Summarizing the content of this section, the notion of non-degeneracy given in Definition \ref{def:nondeg} is the one to prefer when dealing with functionals defined on a Banach space or manifold, because it is tailored 
for the construction of Morse homology (after possibly addressing the subtleties related to stable and unstable manifolds being true submanifolds, see Footnote 5). The drawback is that such a definition seems a priori hard to verify in concrete cases and it is in general unrelated to other notions of non-degeneracy which are easier to check and known to hold generically. With this in mind, in the next section we will show that, for the class of functionals in \eqref{eq:functionalf}, injectivity of the second differential at a critical point implies that the critical point is non-degenerate in the sense of Definition \ref{def:nondeg}. 


\section{Injectivity and non-degeneracy}
\label{sec:hyperbolicoperator}

Let $X:=W^{1,p}_0(\Omega)$ with $\Omega\subset \R^n$ open bounded set and $p\in (2,n]$, and let $f:X\to \R$ be a functional as in \eqref{eq:functionalf}, with $\Psi$ and $g$ satisfying the conditions ($\Psi)$ and ($g$).  
 If $\bar u\in X$ is a critical point of $f$, then a straightforward computation shows that 
\begin{equation}
\label{eq:df2}
\mathrm d^2f(\bar u)[v,w] = \int_\Omega \Psi''(\nabla \bar u)[\nabla v,\nabla w]\, \mathrm dx - \int_\Omega \partial_s g(x,\bar u(x)) v(x)w(x)\, \mathrm d x.
\end{equation}
where $\Psi''$ denotes the Hessian of $\Psi$. Hereafter we assume that the second differential at $\bar u$, seen as an operator $\mathrm d^2f(\bar u):X\to X^*$, is injective. 

\begin{rmk}
Adapting the argument in \cite{Cingolani:2007} one sees that, after a generic perturbation of the function $g$, we can assume that the second differential be injective at every critical point $\bar u$ of $f$. We shall stress that the proof of this fact does not follow from the Sard-Smale theorem because $\mathrm d^2f(\bar u)$ is not a Fredholm operator. The proof is instead based on a finite dimensional reduction argument in a neighborhood of the set of critical points which allows us to apply the classical Sard's lemma in a finite dimensional setting. For future applications it would be interesting to see if such a genericity statement can be obtained directly at the Banach space level. First steps in this direction can be found e.g. in \cite{Lerario:2024}.
\end{rmk}

\begin{thm}
Assume that $f$ is of the form \eqref{eq:functionalf} with $\Psi$ and $g$ satisfying $(\Psi)$ and $(g)$. If $\bar u \in X$ is a critical point of $f$ such that $\mathrm d^2f(\bar u):X\to X^*$ is injective, then $\bar u$ is non-degenerate in the sense of Definition \ref{def:nondeg}.
\label{thm:nondeg}
\end{thm}

Hereafter, to ease the notation, whenever no confusion can arise we use the notation 
 $``\lesssim''$ and $``\gtrsim''$  for inequalities which hold up to some multiplicative constant $c>0$ independent of any of the quantities involved.
Following \cite{Cingolani:2003} and \cite{Asselle:2024}, we see that elliptic regularity implies that $\bar u \in C^1(\overline \Omega)$. This, together with Assumption ($\Psi$), allows us to define a scalar product on $C^\infty_0(\Omega)$ by 
$$\langle v,w\rangle_{\mathbb H} := \int_\Omega \Psi''(\nabla \bar u ) \big [\nabla v,\nabla w\big] \, \mathrm d x$$
and an Hilbert space $\mathbb H$ by taking the completion of $C^\infty_0(\Omega)$ with respect to the induced norm $\|\cdot \|_{\mathbb H}$. Clearly, $\mathbb H$ is isomorphic to $W^{1,2}_0(\Omega)$ and thus we have a continuous embedding $X\hookrightarrow \mathbb H$. By construction, $\mathrm d^2 f(\bar u)$ extends (after identifying $\mathbb H^*$ with $\mathbb H$ using Riesz' representation theorem) to an operator $H: \mathbb H\to \mathbb H$ which is a compact perturbation of the identity and as such is a Fredholm operator with Fredholm index zero. Moreover, the spectrum of $H$ is real and consists of eigenvalues different from 1 (with finite multiplicity) which accumulate to 1. Applying elliptic regularity once again, we obtain an orthogonal decomposition 
$$\mathbb H = \mathbb H^- \oplus \mathbb H^+$$
where $\mathbb H^\pm$ are the positive and negative eigenspace of $H$, and $\mathbb H^-\subset X$. Consequently, this yields a splitting 
$$X = X^-\oplus X^+ = \mathbb H^- \oplus (\mathbb H^+ \cap X).$$
By construction we further have 
\begin{equation}
\mathrm d^2 f(\bar u ) [v^-,v^-] \lesssim -  \|v^-\|^2_{\mathbb H}, \quad \forall v^- \in X^-,
\label{eq:v-}
\end{equation}
and 
\begin{equation}
\mathrm d^2 f(\bar u ) [v^+,v^+] \gtrsim  \|v^+\|^2_{\mathbb H}, \quad \forall v^+ \in X^+.
\label{eq:v+}
\end{equation}
Finite dimensionality of $X^-$ implies that $\|\cdot\|_{\mathbb H}$ is equivalent to $\|\cdot \|$ on $X^-$ and hence the $C^2$-regularity of $f$ yields that \eqref{eq:v-} is still satisfied (possibly with a smaller multiplicative constant) in a neighborhood of $\bar u$. In contrast, there is a priori no reason why \eqref{eq:v+} should hold in a neighborhood of $\bar u$ (see for instance the example in the introduction). However, if it does, then one can argue as in \cite{Asselle:2024} to show that $\bar u$ is non-degenerate in the sense of Definition \ref{def:nondeg}. 
To this purpose, we set
$$p^*_{C^2}:= \frac{2p^*}{n}.$$
We readily see that $ p^*_{C^2} < p^*-2$ for every $n\geq 3$.

\begin{lem}
Assume that $f$ is of the form \eqref{eq:functionalf} with $\Psi$ and $g$ satisfying $(\Psi)$ and $(g)$. Assume further that $(g)$ is satisfied for some $0\leq q<p^*_{C^2}$. Then, there exists $\delta>0$ such that 
\begin{equation}
\mathrm d^2 f(u ) [v^+,v^+] \gtrsim \|v^+\|^2_{\mathbb H}, \quad \forall v^+ \in X^+, \ \forall \|u-\bar u\|<\delta.
\label{eq:uv+}
\end{equation}
\label{lem:p*c2}
\end{lem}

\begin{proof}[Proof of Lemma \ref{lem:p*c2}]
We follow closely the proof of Lemma 4.2 in \cite{Cingolani:2003} and assume by contradiction that there exist sequences $\{z_n\}\subset X$ and $\{v_n\}\subset X^+$ with $\|v_n\|_{\mathbb H}=1$ such that $z_n\to \bar u$ and 
$$\liminf_{n\to +\infty} \mathrm\  d^2 f (z_n)[v_n,v_n]\leq 0.$$
Without loss of generality we can assume that $v_n\rightharpoonup v$ in $W^{1,2}_0(\Omega)$, for some $v\in \mathbb H^+$, and hence $v_n\to v$ in $L^r(\Omega)$ for every $r<2^* = \frac{2n}{n-2}$. Also, up to passing to a further subsequence if necessary, we can assume that $z_n \to \bar u$ and $v_n\to v$ pointwise almost everywhere. 

\vspace{2mm}

\textbf{Claim 1.} $\displaystyle \int_\Omega \partial_s g(x,z_n(x)) v_n^2 (x) \mathrm dx \to \int_\Omega \partial_s g(x,\bar u (x)) v^2(x)\, \mathrm d x, \quad \text{for}\ n \to +\infty.$

\noindent Condition ($g$) implies that 
\begin{align*}
\int_\Omega\big |\partial_s g(x,z_n(x)) v_n^2 (x)\big | \, \mathrm dx \lesssim \int_\Omega |z_n|^{q}|v_n|^2\, \mathrm d x + \int_\Omega |v_n|^2 \, \mathrm d x.
\end{align*}
As the second integral on the RHS clearly converges, we only focus on the first one. In order to apply Lebesgue dominated convergence theorem we just have to show that 
$$\int_\Omega  |z_n|^{q}|v_n|^2\, \mathrm d x \lesssim 1, \quad \forall n \in \N.$$ 
H\"older inequality for conjugated exponents $1<s,r<+\infty$ implies that 
$$ \int_\Omega  |z_n|^{q}|v_n|^2\, \mathrm d x \leq \| z_n\|_{qr}^{q_2}\cdot \|v_n\|_{2s}^2.$$ 
Therefore, we only have to ensure that $qr \leq p^*$ and $2s \leq 2^*$, which amounts to 
$$\left \{\begin{array}{l} r \leq \displaystyle \frac{p^*}{q}, \\ \displaystyle \frac{r}{r-1}\leq \frac{n}{n-2}.\end{array}\right.$$ 
As $r\mapsto \frac r{r-1}$ is strictly monotonically decreasing for $r>1$, all we have to check is that the second inequality is satisfied for $r = p^*/q$. A straightforward computation shows now that this is equivalent to $q\leq p^*_{C^2}$.  

\vspace{2mm}

\textbf{Claim 2.} $v \neq 0$. 

\noindent Assuming by contradiction that $v=0$ and using \eqref{eq:df2}, ($\Psi$), and Claim 1 we obtain 
\begin{align*}
\mathrm d^2f(\bar z_n)[v_n,v_n] &= \int_\Omega \Psi''(\nabla z_n )[\nabla v_n,\nabla v_n]\, \mathrm dx - \int_\Omega \partial_s g(x, z_n (x)) v_n^2(x)\, \mathrm d x\\ 
			&\gtrsim \|v_n\|_{\mathbb H}^2 - \int_\Omega \partial_s g(x, z_n (x)) v_n^2(x)\, \mathrm d x	\\
			&\gtrsim 1
\end{align*}
which is a contradiction. 

\vspace{2mm}

\textbf{Claim 3.} $H[v,v]\leq 0$. 

\noindent Clearly, Claim 3 contradicts Claim 2 because $H$ is positive definite on $\mathbb H^+$ thus finishing the proof. To prove the claim we just notice that we can apply Theorem 4.1 in \cite{Cingolani:2003} to infer that 
$$\int_\Omega \Psi'' (\nabla \bar u)[\nabla v,\nabla v]\, \mathrm d x \leq \liminf_{n\to +\infty} \int_\Omega \Psi''(\nabla z_n)[\nabla v_n,\nabla v_n]\, \mathrm d x.$$
Hence, Claim 1 implies that $\displaystyle H[v,v] \leq \liminf_{n\to +\infty} \mathrm d^2 f (z_n)[v_n,v_n] \leq 0.$
\end{proof}

\begin{rmk}
Equation \eqref{eq:uv+} can be interpreted as a sort of local uniform convexity of $f$ in the direction $X^+$. Notice that assuming that $(g)$ hold for some $0\leq q< p^*_{C^2}$ implies that the lower order part of $f$ in \eqref{eq:functionalf} is of class $C^2$ on the Hilbert extension $\mathbb H$. 
\end{rmk}

If ($g$) is satisfied for some $q\in [p^*_{C^2},p^*-2)$, then we cannot hope Lemma \eqref{lem:p*c2} to hold. Indeed, as shown in Lemma 4.4 in \cite{Cingolani:2003}, Equation \eqref{eq:uv+} holds only on $L^\infty$-balls, with the constant $\delta$ (namely, the size of the neighborhood of $\bar u$) going to zero as the radius of the $L^\infty$-balls goes to infinity. Nevertheless, the desired hyperbolic operator still exists as we now show. To ease the notation we assume hereafter that $\bar u=0$. 

We fix a neighborhood $\mathcal U$ of $\bar u =0$ such that 
\begin{equation}
\mathrm d^2 f(u) [v^-,v^-] \leq - 2c\,  \|v^-\|^2, \quad \forall v^- \in X^-, \ \forall u\in \mathcal U,
\label{eq:uv-2}
\end{equation}
for some constant $c>0$. With slight abuse of notation, $\|\cdot \|$ denotes here any (equivalent) norm on $X^-$. 

\begin{lem}
Assume that $f$ is of the form \eqref{eq:functionalf} with $\Psi$ and $g$ satisfying $(\Psi)$ and $(g)$. Then, there exist $\delta>0$ and $c_1>0$ such that 
\begin{equation}
\mathrm d^2 f(u) [u^+,u^+] \geq c_1 \, \|u^+\|^2_{\mathbb H} - c \, \|u^-\|^2,\quad \forall \|u\|<\delta,
\label{eq:uv+2}
\end{equation}
where $c>0$ is the constant in \eqref{eq:uv-2} and $u=u^-+u^+$ is the decomposition of $u$ in the splitting $X=X^-+X^+$. 
\label{lem:p*-2}
\end{lem}

\begin{rmk}
Equation \eqref{eq:uv+2} is a much weaker requirement than the local uniform convexity condition \eqref{eq:uv+}. First, for any fixed $u$ it is a condition for $\mathrm d^2f(u)$ on only one specific vector in $X^+$, namely $u^+$, and not on the whole $X^+$. Second, it allows the quantity $\mathrm d^2f(u)[u^+,u^+]$ to be negative, but the negative contribution has then to be controlled by $\mathrm d^2f(u)[u^-,u^-]$. As such, \eqref{eq:uv+2} is potentially satisfied for a large class of functionals which do not satisfy \eqref{eq:uv+}, and hence might allow us to define Morse homology for functionals which are relevant for applications but do not satisfy the local uniform convexity property. Functionals of the form \eqref{eq:functionalf} for which ($g$) is satisfied for some $p^*_{C^2}\leq q<p^*-2$ represent the first instance of this fact.  
\end{rmk}

\begin{proof}[Proof of Theorem \ref{thm:nondeg}]
We set $L:=(-\text{id}_{X^+},\text{id}_{X^-})$. Recalling that for simplicity we assumed $\bar u =0$, all we have to show is that 
$$\frac{\mathrm d}{\mathrm d t}\Big |_{t=0} f (e^{tL} u) < 0, \quad \forall \|u\|<\delta, \ u \neq 0,$$
for some $\delta>0$ small enough. Using \eqref{eq:uv-2}, \eqref{eq:uv+2}, and a Taylor expansion with integral remainder we compute 
\begin{align*}
\frac{\mathrm d}{\mathrm d t}\Big |_{t=0} f (e^{tL} u) &= \mathrm df (u)[Lu]\\
			&= \int_0^1 \mathrm d^2 f (su)[Lu,u]\, \mathrm d s\\
			&= \int_0^1 \mathrm d^2 f(su)[u^-,u^-]\, \mathrm ds - \int_0^1 \mathrm d^2 f(su)[u^+u^+]\, \mathrm ds \\
			&\leq -2c \|u^-\|^2 - c_1 \|u^+\|_{\mathbb H}^2 + c \|u^-\|^2 \\
			&\leq - c\|u^-\|^2 - c_1 \|u^+\|_{\mathbb H}^2,
\end{align*}
thus finishing the proof.
\end{proof}

Abstracting from the concrete situation of Lemma \ref{lem:p*-2}, we obtain the following criterion of independent interest to check non-degeneracy.  

\begin{thm}
Let $(Y,\|\cdot\|)$ be a Banach space, $\varphi:Y\to \R$ be a $C^2$-functional, and $y=0$ be a critical point of $\varphi$. Assume there exist a splitting $Y=Y^+\oplus Y^-$, with $\dim Y^-<+\infty$, and a constant $c>0$ such that 
\begin{equation}
\pm \mathrm d^2 \varphi (0)[z,z] \geq 2c \|z\|_{w}^2, \ \forall z \in Y^\pm,
\label{eq:splittingintro}
\end{equation}
for some weaker\footnote{Since $\dim Y^-<+\infty$, we could replace $\|\cdot\|_w$ with the Banach norm in the part of Equations  \eqref{eq:splittingintro} and \eqref{eq:splittingintro2} involving $Y^-$.} norm $\|\cdot \|_w$. 
If there exists $\delta>0$ such that 
\begin{equation}
\mathrm d^2 \varphi(y)[y^+,y^+] \geq c \|y^+\|_w^2 - c \|y^-\|_w^2, \quad \forall y=y^++y^- \ \text{s.t.} \ \|y\|<\delta,
\label{eq:splittingintro2}
\end{equation}
then $y_0$ is a non-degenerate critical point of $\varphi$ in the sense of Definition \ref{def:nondeg}. 
\label{thm:criterionnondeg}
\end{thm}

\begin{rmk}
The criterion given in Theorem \ref{thm:criterionnondeg} can be generalized to the case in which both $Y^+$ and $Y^-$ are infinite dimensional by requiring that \eqref{eq:splittingintro2} holds also swapping the roles of $y^+$ a $y^-$ and changing $\geq$ to $\leq$. This allows us to check non-degeneracy in many cases of interests where the Morse index is infinite\footnote{This is of great importance because the classical approach to Morse theory based on deformation of sublevel sets and cell attachments fails when the Morse index is infinite, being the infinite dimensional sphere contractible.}, for instance 
for solutions of systems of quasi-linear elliptic pde's of the form 
\begin{equation}
\label{eq:bvpsystems}
\left \{ \begin{array}{r} - \text{div}\, \big (\nabla \Psi (\nabla u)\big ) = \partial_u G(x, u,v) \quad \quad  \text{in}\ \Omega, \\ \text{div}\, \big (\nabla \Psi (\nabla v)\big )= \partial_v G(x,u,v) \quad \quad  \text{in}\ \Omega, \\ u=v=0 \qquad \qquad  \qquad \ \ \ \text{on}\ \partial \Omega,\end{array}\right .
\end{equation}
or for perturbed $\alpha$-Dirac-harmonic maps, see \cite{Jost:2021} and references therein. We refrain to do it here because the construction of Morse homology is much more delicate in case of infinite Morse indices and subtle 
technical aspects (such as the pre-compactness of the intersection between stable and unstable manifolds of pairs of critical points) still needs to be addressed in a Banach setting, and rather leave it  for future research. 
\end{rmk}

\begin{proof}[Proof of Lemma \ref{lem:p*-2}]
By Assumption ($g$), there exist constants $\alpha,\beta>0$ such that
$$|\partial_s g(x,s)| \leq \alpha + \beta |s|^{p^*-2}, \quad \forall x \in \Omega, \ \forall s \in \R.$$
For $\epsilon>0$ to be fixed later we set 
$$t_\epsilon(u) := \frac \epsilon p \int_\Omega |\nabla u|^p \, \mathrm dx - \frac{\beta}{p^*(p^*-1)} \int_\Omega |u|^{p^*}\, \mathrm d x$$
and notice that there exist $\delta=\delta(\epsilon)>0$ and $\epsilon'=\epsilon'(\epsilon)>0$ small enough such that 
$$t_\epsilon (u) \geq \epsilon' \int_\Omega |\nabla u|^p \, \mathrm d x = \epsilon' \, \|\nabla u\|_p^p, \quad \forall \|u\|<\delta,$$
hence for a possibly smaller constant $\epsilon''=\epsilon''(\epsilon)>0$
\begin{equation}
\mathrm d^2 t_\epsilon (u)[u,u] \geq \epsilon''\, \|\nabla u \|_p^p, \quad \forall \|u\|<\delta.
\label{eq:convexitytepsilon}
\end{equation}
We finally set $\tilde f:X\to \R$ by 
$$\tilde f(u) := f(u)-t_\epsilon (u), \quad  \forall u \in X.$$
Adapting the argument in \cite[Lemma 4.5]{Cingolani:2003} yields for a possibly smaller $\delta>0$ (the details are left to the reader) 
\begin{equation}
\mathrm d^2\tilde f(u)[v^+,v^+] \geq 2 c_1 \,  \|v^+\|_{\mathbb H}^2, \ \forall v^+ \in X^+, \ \forall \|u\|<\delta,
\label{eq:d2tildef}
\end{equation}
for some constant $c_1>0$, that is, $\tilde f$ is locally uniformly convex in the $X^+$-direction. 

We now compute using \eqref{eq:convexitytepsilon} and \eqref{eq:d2tildef}
\begin{align}
\mathrm d^2f(u)& [u^+,u^+] \nonumber  \\ 
			& = \mathrm d^2 \tilde f(u)[u^+,u^+] + \mathrm d^2 t_\epsilon (u)[u^+,u^+] \nonumber \\ 
			&= \mathrm d^2 \tilde f(u)[u^+,u^+] + \mathrm d^2 t_\epsilon (u)[u,u]  - \mathrm d^2 t_\epsilon (u)[u^-,u^-] - 2 \mathrm d^2 t_\epsilon (u)[u^-,u^+]  \label{eq:df2estimate} \\
			& \geq 2 c_1 \,  \|v^+\|_{\mathbb H}^2 +  \epsilon''\, \|\nabla u \|_p^p - \underbrace{\mathrm d^2 t_\epsilon (u)[u^-,u^-]}_{=:(\ast)} - 2 \underbrace{\mathrm d^2 t_\epsilon (u)[u^-,u^+]}_{=:(\ast \ast)}, \nonumber
\end{align}
and estimate $(\ast)$ and $(\ast \ast )$ separately. For $(\ast)$ we have 
\begin{align}
|\mathrm d^2 t_\epsilon (u)& [u^-,u^-]| \nonumber \\ &\leq \epsilon \underbrace{\int_\Omega |\nabla u|^{p-2} |\nabla u^-|^2 \, \mathrm d x}_{=:(\ast_1)} + \epsilon(p-2)  \underbrace{\int_\Omega |\nabla u|^{p-4} |\langle \nabla u, \nabla u^-\rangle |^2 \, \mathrm d x}_{=:(\ast_2)} + \beta \underbrace{\int_\Omega |u|^{p^*-2}|u^-|^2\, \mathrm d x}_{=:(\ast_3)}.
\label{eq:d2teu-}
\end{align}
For $(\ast_1)$ we obtain applying the elementary inequality 
$$(a+b)^{p-2} \leq c_2 (a^{p-2}+ b^{p-2}), \quad \forall a,b >0,$$
for some constant $c_2>0$ depending only on $p>2$, H\"older's inequality with conjugated exponents $r=\frac p{p-2}$ and $s = \frac p2$, and subsequently Young's inequality with $\frac 1 \nu + \frac 1 \mu =1$ for some $\nu$ such that $\nu(p-2) > p$: 
\begin{align*}
(\ast_1) &\leq c_2 \Big (\int_\Omega |\nabla u^+|^{p-2} |\nabla u^-|^2 \, \mathrm d x + \int_\Omega |\nabla u^-|^p\, \mathrm d x	\Big )	\\
 		& \leq c_2 \Big ( \|\nabla u^+\|_p^{p-2} \|\nabla u^-\|_p^2 + \|\nabla u^-\|_p^p\Big )\\
		& \leq c_2 \Big (\frac 1 \nu \|\nabla u^+\|_p^{\nu(p-2)} + \frac 1 \mu \|\nabla u^-\|_p^{2\mu} + \|\nabla u^-\|_p^p\Big )\\ 
		& \leq c_2 \Big (\|\nabla u^+\|_p^{\nu(p-2)} + \|\nabla u^-\|_p^{2\mu} + \|\nabla u^-\|_p^p\Big ).
\end{align*}
$(\ast_2)$ can be estimated in an analogous way and leads to an estimate identical to that for $(\ast_1)$. The details are left to the reader. For $(\ast_3)$ we compute in a similar fashion using the elementary inequality 
$$(a+b)^{p^*-2} \leq c_3 (a^{p^*-2}+ b^{p^*-2}), \quad \forall a,b >0,$$
for some constant $c_3>0$ depending only on $p>2$, H\"older's inequality with conjugated exponents $r=\frac{p^*}{p^*-2}$ and $s = \frac{p^*}2$, and subsequently Young's inequality with $\frac 1 i + \frac 1 j =1$ for some $i$ such that $i(p^*-2) > p$: 
\begin{align*}
(\ast_3) &\leq c_3 \Big (\int_\Omega |u^+|^{p^*-2} |u^-|^2 \, \mathrm d x + \int_\Omega |u^-|^{p^*}\, \mathrm d x	\Big )	\\
 		& \leq c_3 \Big ( \|u^+\|_{p^*}^{p^*-2} \|u^-\|_{p^*}^2 + \|u^-\|_{p^*}^{p^*}\Big )\\
		& \leq c_3 \Big (\|u^+\|_{p*}^{i(p^*-2)} + \|u^-\|_{p^*}^{2j} + \|u^-\|_{p^*}^{p^*}\Big ).
\end{align*}
Up to shrinking $\delta>0$ further if necessary, putting all estimates above in \eqref{eq:d2teu-} and using the fact that all norms are equivalent on $X^-$ we obtain for some constant $c_4>0$:
\begin{align}
|\mathrm d^2 t_\epsilon (u)& [u^-,u^-]| \nonumber \\ 
		&\leq \epsilon c_2 (p -1 ) \Big (\|\nabla u^+\|_p^{\nu(p-2)} + \|\nabla u^-\|_p^{2\mu} + \|\nabla u^-\|_p^p\Big ) + c_3 \Big (\|u^+\|_{p*}^{i(p^*-2)} + \|u^-\|_{p^*}^{2j} + \|u^-\|_{p^*}^{p^*}\Big )\nonumber \\
		& \leq c_4 \Big (\epsilon \|\nabla u^+\|_p^{\nu(p-2)} + \epsilon \|u^-\|^{\mu_1} + \|u^+\|_{p^*}^{i(p^*-2)} + \|u^-\|^{j_1}\Big ) \label{eq:d2teu-2}\\
		&\leq c_4 \Big (\epsilon \|\nabla u^+\|_p^{\nu(p-2)} + \epsilon \|u^-\|^{\mu_1} + \|\nabla u^+\|_{p}^{i(p^*-2)} + \|u^-\|^{j_1}\Big ),  \quad \forall \|u\|<\delta, \nonumber
\end{align}
where $\mu_1 := \min \{2\mu,p\}>2$ and $j_1 := \min\{2j,p^*\}>2$, and in the last inequality we have used the Sobolev embedding theorem.

We now move on to $(\ast \ast)$: 
\begin{align}
|\mathrm d^2 t_\epsilon (u) [u^+,u^-]| &\leq \epsilon \underbrace{\int_\Omega |\nabla u|^{p-2} |\langle \nabla u^+,\nabla u^-\rangle|^2 \, \mathrm d x}_{=:(\ast \ast_1)} \nonumber \\ 
			 & + \epsilon(p-2)  \underbrace{\int_\Omega |\nabla u|^{p-4} |\langle \nabla u, \nabla u^+\rangle | |\langle \nabla u, \nabla u^-\rangle | \, \mathrm d x}_{=:(\ast \ast_2)} + \beta \underbrace{\int_\Omega |u|^{p^*-2}|u^+| |u^-|\, \mathrm d x}_{=:(\ast \ast_3)}.
\label{eq:d2teu+}
\end{align}
As above $(\ast \ast_1)$ and $(\ast \ast_2)$ can be estimated in a similar way. Arguing as for $(\ast_1)$, we compute for $(\ast \ast_1)$ for some constant $c_5>0$ depending only on $p>2$: 
\begin{align}
(\ast \ast_1) &\leq c_5 \Big ( \int_\Omega |\nabla u^+|^{p-1} |\nabla u^-| \, \mathrm d x + \int_\Omega |\nabla u^-|^{p-1} |\nabla u^+|\, \mathrm d x \nonumber \\ 	
		 &\leq c_5 \Big (\|\nabla u^+\|_p^{p-1}\|\nabla u^-\|_p + \|\nabla u^-\|_{2p-2}^{p-1}\|\nabla u^+\|_2\Big )\\
		 &\leq c_5 \Big (\|\nabla u^+\|_p^{2p-2} + \|\nabla u^-\|_p^2 + \|\nabla u^-\|_{2p-2}^{2p-2} + \|\nabla u^+\|_2^2\Big ) \nonumber \\
		 &\leq c_5 \Big (\|\nabla u^+\|_p^{2p-2} + \|u^-\|^2 + \|u^-\|^{2p-2} + \|u^+\|_{\mathbb H}^2\Big ), \nonumber 
\end{align}
where in the second inequality we have used both H\"older's inequality with conjugated exponents $r= \frac p{p-1}$ and $s=p$ and the Cauchy-Schwartz inequality, and in the last inequality we have used the elementary inequality $ab\leq a^2+b^2$ for all $a,b>0$ and the fact that $\|\cdot\|_{\mathbb H}$ is equivalent to the $W^{1,2}$-norm. Arguing as for $(\ast_3)$, we compute for $(\ast \ast_3)$ for some constant $c_6>0$ depending only on $p>2$: 
\begin{align*}
(\ast \ast_3) &\leq c_6  \Big (\int_\Omega |u^+|^{p^*-1} |u^-| \, \mathrm d x + \int_\Omega |u^+| |u^-|^{p^*-1}\, \mathrm d x	\Big )	\\
 		& \leq c_6 \Big ( \|u^+\|_{p^*}^{p^*-1} \|u^-\|_{p^*} + \|u^+\|_2 \|u^-\|_{2p^*-1}^{p^*-1}\Big )\\
		& \leq c_6 \Big (\|u^+\|_{p^*}^{\gamma(p^*-1)} + \|u^-\|_{p^*}^\rho + \|u^-\|_{2p^*-1}^{\gamma(p^*-1)} + \|u^+\|_2^\rho\Big )\\
		&\leq c_6 \Big ( \|\nabla u^+\|_p^{\gamma(p^*-1)} + \|u^-\|^\rho + \|u^-\|^{\gamma(p^*-1)} + \|u^+\|_{\mathbb H}^\rho\Big ),
\end{align*}
where in the second inequality we have used both H\"older's inequality with conjugated exponents $r=\frac{p^*}{p^*-1}$ and $s=p^*$ and the Cauchy-Schwartz inequality, in the third inequality we have used Young's inequality with $\rho>2$ so close to $2$ that $\gamma(p^*-1)>p$, and in the last inequality we have used the fact that all norms on $X^-$ are equivalent, the Sobolev embedding theorem, and Poincare's inequality. Putting all estimates together in \eqref{eq:d2teu+} yields for some constant $c_7>0$ after possibly shrinking $\delta>0$ further: 
\begin{align}
|\mathrm d^2 t_\epsilon (u)[u^+,u^-]| &\leq  \epsilon c_5 (p -1 ) \Big (\|\nabla u^+\|_p^{2p-2} + \|u^-\|^2 + \|u^-\|^{2p-2} + \|u^+\|_{\mathbb H}^2 \Big ) \nonumber \\ 
& + c_6 \Big (\|\nabla u^+\|_{p}^{\gamma (p^*-1)} + \|u^-\|^\rho  + \|u^-\|^{\gamma(p^*-1)} + \|u^+\|_{\mathbb H}^\rho\Big )\label{eq:d2teu+2} \\
		& \leq c_7 \Big (\epsilon \|\nabla u^+\|_p^{2p-2} +  \epsilon \|u^+\|_{\mathbb H}^2 + \epsilon \|u^-\|^{2} + \|\nabla u^+\|_{p}^{\gamma (p^*-1)} + \|u^-\|^{j_2} + \|u^+\|_{\mathbb H}^\rho\Big ) \nonumber
\end{align}
where $j_2:=\min \{\rho, \gamma(p^*-1)\}>2$. Plugging \eqref{eq:d2teu-2} and \eqref{eq:d2teu+2} into \eqref{eq:df2estimate}, using the elementary inequality 
$$(a+b)^p \leq 2^p (a^p + b^p), \quad \forall a,b>0,$$
and again the fact that all norms on $X^-$ are equivalent, we get 
\begin{align*}
\mathrm d^2 f (u) & [u^+,u^+] \\ 
	& \geq 2 c_1 \,  \|v^+\|_{\mathbb H}^2 +  \epsilon''\, \|\nabla u \|_p^p - \mathrm d^2 t_\epsilon (u)[u^-,u^-]- 2 \mathrm d^2 t_\epsilon (u)[u^-,u^+]\\
	& \geq 2 c_1 \,  \|v^+\|_{\mathbb H}^2 +  \epsilon''\, \|\nabla u \|_p^p - c_4 \Big (\epsilon \|\nabla u^+\|_p^{\nu(p-2)} + \epsilon \|u^-\|^{\mu_1} + \|\nabla u^+\|_{p}^{i(p^*-2)} + \|u^-\|^{j_1}\Big )\\
	& \ \ \ - c_7 \Big (\epsilon \|\nabla u^+\|_p^{2p-2} +  \epsilon \|u^+\|_{\mathbb H}^2 + \epsilon \|u^-\|^{2} + \|\nabla u^+\|_{p}^{\gamma (p^*-1)} + \|u^-\|^{j_2} + \|u^+\|_{\mathbb H}^\rho\Big ) \\
	&\geq  2 c_1 \,  \|v^+\|_{\mathbb H}^2 - c_7\epsilon \|u^+\|_{\mathbb H}^2 - c_7\epsilon \|u^-\|^2 +  \epsilon''\, \|\nabla u \|_p^p \\ 
	&\ \ \ - 2^p \big ( c_4\epsilon \|\nabla u\|_p^{\nu(p-2)} + c_4 \|\nabla u\|_p^{i(p^*-2)} + c_7 \epsilon \|\nabla u\|_p^{2p-2} + c_7 \|\nabla u\|_p^{\gamma(p^*-1)}\big ) \\
	& \ \ \ - 2^p \big ( c_4\epsilon \|u^-\|^{\nu(p-2)} + c_4 \|u^-\|_p^{i(p^*-2)} + c_7 \epsilon \|u^-\|^{2p-2} + c_7 \|u^-\|^{\gamma(p^*-1)}\big )\\
	& \ \ \ - c_4 \epsilon \|u^-\|^{\mu_1} - c_4 \|u^-\|^{j_1}  - c_7\|u^-\|^{j_2} - c_7 \|u^+\|_{\mathbb H}^\rho. 
\end{align*}
Choosing now $\epsilon< \min \{\frac{c_1}{2 c_7},\frac{c}{2c_7}\}$, where $c>0$ is given by \eqref{eq:uv-2}, yields 
\begin{align*}
\mathrm d^2 f (u) & [u^+,u^+] \\ 
	& \geq \frac 32 c_1 \|u^+\|_{\mathbb H}^2  - \frac c2 \|u^-\|^2 - c_7\|u^+\|_{\mathbb H}^\rho  \\
	& \ \ \ + \underbrace{\epsilon'' \|\nabla u\|_p^ p - 2^p \big ( c_4\epsilon \|\nabla u\|_p^{\nu(p-2)} + c_4 \|\nabla u\|_p^{i(p^*-2)} + c_7 \epsilon \|\nabla u\|_p^{2p-2} + c_7 \|\nabla u\|_p^{\gamma(p^*-1)}\big )}_{=:(1)} \\
	& \ \ \ - \underbrace{2^p \big ( c_4\epsilon \|u^-\|^{\nu(p-2)} + c_4 \|u^-\|_p^{i(p^*-2)} + c_7 \epsilon \|u^-\|^{2p-2} + c_7 \|u^-\|^{\gamma(p^*-1)}\big )}_{=:(2)} \\ 
	& \ \ \ - \underbrace{\big (c_4 \epsilon \|u^-\|^{\mu_1} + c_4 \|u^-\|^{j_1}  + c_7\|u^-\|^{j_2}\big )}_{=:(3)},
	\end{align*}
	and since by construction 
	\begin{align*}
	&\rho > 2,\\ 
	& \min \{ \nu(p-2),i(p^*-2), 2p-2, \gamma(p^*-1)\} > p, \\
	& \min \{\nu (p-2),i(p^*-2), 2p-2, \gamma(p^*-1), \mu_1,\j_1,j_2\} >2,
	\end{align*}
	up to shrinking $\delta>0$ further if necessary we get 
	$$c_7\|u^+\|_{\mathbb H}^\rho \leq \frac{c_1}2 \|u^+\|_{\mathbb H}^2, \quad (1)\geq 0, \quad (2)\leq \frac c4 \|u^-\|^2, \quad (3) \leq \frac c4 \|u^-\|^2,$$
	hence finally 
	$$\mathrm d^2 f (u) [u^+,u^+] \geq c_1 \|u^+\|_{\mathbb H}^2 - c \|u^-\|^2$$
	as desired. 
\end{proof}


\section{The Cerami condition}
\label{sec:Ceramicondition}

In this final section we discuss the Cerami condition for the functionals $f:X\to \R$ as in \eqref{eq:functionalf}. Most of the results we present here are already known and can be found in the literature. 
However, we still find convenient to collect them here. 
It is well-known (see \cite{Cingolani:2005,Cingolani:2018,Liu:2010} and references therein) that the key step 
is to show that Cerami sequences are bounded, as then pre-compactness follows adapting standard arguments, see \cite{Benci}. As it turns out, boundedness of Cerami 
(and more generally, Palais-Smale) sequences strongly depends on the growth properties of the function $g$. For convenience we will split Condition $(g)$ into three subconditions:

\vspace{2mm}

($g-$sub) Condition ($g$) is satisfied for some $q<p-2$. 

\vspace{2mm}

($g-$lin) $\partial_s g(\cdot ,s)$ grows exactly as $|s|^{p-2}$.

\vspace{2mm}

($g-$sup) Condition ($g$) is satisfied for some $p-2<q<p^*-2,$ and $\displaystyle \lim_{|s|\to +\infty} |s|^{2-p}|\partial_s g(\cdot,s)|=+\infty$. 

Such subconditions are usually referred to as \textit{sublinear, linear, and superlinear growth at infinity} respectively. In case of sublinear growth at infinity, one 
can show that $f$ is bounded from below and satisfies the Palais-Smale condition. The proof is identical to \cite[Lemma 2.7]{Asselle:2024} and will be omitted. 

\begin{lem}
Assume $f:X\to \R$ as in \eqref{eq:functionalf} satisfies $(\Psi)$ and $(g\mathrm{-sub})$. Then $f$ is bounded from below and satisfies the Palais-Smale condition. 
\end{lem}

We now consider the  case of linear growth at infinity. Here, additional assumptions are needed in order to obtain the Cerami condition. Following \cite{Cingolani:2018} we assume that almost everywhere in $\Omega$
\begin{equation}
\lim_{|s|\to +\infty} \frac{g(\cdot,s)}{|s|^{p-2}s} = \lambda,
\label{eq:growthadditional1}
\end{equation}
for some $\lambda \in \R$. For simplicity we also assume $\Psi=\Psi_\kappa$ for some $\kappa>0$, but the argument should carry over with minor modifications to any general $\Psi$ satisfying ($\Psi$). In what follows $\Delta_p u:= \text{div} (|\nabla u|^{p-2} \nabla u)$ denotes the $p$-Laplace operator.  

\begin{lem}
Assume $f:X\to \R$ as in \eqref{eq:functionalf}, with $\Psi=\Psi_\kappa$ for some $\kappa >0$, satisfies and $(g\mathrm{-lin})$. Assume further that 
 \eqref{eq:growthadditional1} is satisfied for some $\lambda \notin \sigma(-\Delta_p)$. Then $f$ satisfies the Cerami condition. 
\end{lem}

A  proof can be found in \cite[Theorem 1.1]{Cingolani:2018}. The case $\lambda \in \sigma (-\Delta_p)$ is more delicate and in fact additional conditions are needed to ensure the Cerami condition. To keep the exposition as elementary as possible we refrain to treat such a case here and refer to \cite[Theorem 1.2]{Cingolani:2018} instead. 

Finally, in case of superlinear growth at infinity, following \cite{Liu:2010} we assume that there exists $r>0$ such that, for almost every $x\in \Omega$, the function
\begin{equation}
 s\mapsto \frac{g(\cdot,s)}{|s|^{p-2}s} 
\label{eq:growthadditional2}
\end{equation}
is monotonically increasing\footnote{Notice that ($g-$sup) implies that $\frac{g(\cdot,s)}{|s|^{p-2}s}  \to +\infty$ as $|s|\to +\infty$.} for $s\geq r$ and monotonically decreasing for $s\leq -r$, and that 
\begin{equation}
G(\cdot,s) \geq - \alpha |s|^p, \quad \forall s \in \R
\label{eq:growthadditional3}
\end{equation}
for some $\alpha>0$. The proof of the next lemma is obtained from \cite[Lemma 2.5]{Liu:2010} with some minor modifications and will be omitted. 

\begin{lem}
Let $f:X\to \R$ be as in \eqref{eq:functionalf} and satisfy $(\Psi)$ and $(g\mathrm{-sup})$. Assume further that \eqref{eq:growthadditional2} and \eqref{eq:growthadditional3} hold. Then $f$ satisfies the Cerami condition. 
\end{lem}

\begin{rmk}
A standard assumption to ensure that $f$ satisfies the Palais-Smale condition in case of superlinear growth is the so-called \textit{Ambrosetti-Rabinowitz} condition, namely that there exist $\mu>p$ and $R>0$ such that 
$$0<\mu G(\cdot,s) \leq g(\cdot,s)s, \quad \forall |s|\geq r.$$ 
This is not assumed here. Hence, for certain non-linearities, $f$ might fail to satisfy the Palais-Smale condition but will satisfy Cerami instead.  
\end{rmk}


\bibliography{_biblio}
\bibliographystyle{plain}

\end{document}